\documentclass[12pt,a4paper]{amsart}
\allowdisplaybreaks % for proof reading

\usepackage{color}
\usepackage{amsmath}
\usepackage{amssymb}
\usepackage{amsthm}
\usepackage{amsfonts}
\usepackage{latexsym}
\usepackage{hyperref}

\theoremstyle{plain}
\newtheorem{theorem}{Theorem}[section]
\newtheorem{corollary}[theorem]{Corollary}
\newtheorem{lemma}[theorem]{Lemma}
\newtheorem{proposition}[theorem]{Proposition}
\theoremstyle{definition}

\newtheorem{remark}[theorem]{Remark}
\numberwithin{equation}{section}
\newcommand\pref[1]{~(\ref{#1})}
\newcommand\op[1]{#1}
\newcommand\cA{\mathcal A}

\newcommand\cF{\mathcal F}
\newcommand\cG{\mathcal G}

\newcommand\cS{\mathcal S}

\newcommand\cD{\mathcal D}

\newcommand\C{\mathbb C}
\newcommand\R{\mathbb R}
\newcommand\N{\mathbb N}

\newcommand\RE {\text{\rm Re}\,}
\newcommand\IM {\text{\rm Im}\,}

\newcommand\supp {\text{\rm supp\,}}

\newcommand\al{\alpha}

\newcommand\eps{\varepsilon}
\newcommand\la{\lambda}
\newcommand\vp{\varphi}

\newcommand\dimH{n}
\newcommand\Hn{{H_\dimH}}

\newcommand\Un {{\text{U}(\dimH)}}  %%% gruppo unitario
\def\algop{{\mathbb D}_{\text{\rm rad}}}
\def\crad{\cD_{\text{\rm rad}}}

\newcommand\fan{{\Sigma^*}}  %ventaglio aperto
\newcommand\cfan{\Sigma}   %ventaglio chiuso
\newcommand\schwkrad{{\mathcal S}_{\text{\rm rad}}(\Hn)}
\newcommand\schwkradp{{\mathcal S}'_{\text{\rm rad}}(\Hn)}
\newcommand\Lunorad{L^1_{\text{\rm rad}}(\Hn)}
\newcommand\Lduerad{L^2_{\text{\rm rad}}(\Hn)}
\newcommand\dualH[2]{\langle #1,#2 \rangle_{\Hn}} %dualita su Heis
\newcommand\bigdualH[2]{\left\langle #1,#2 \right\rangle_{\Hn}} %dualita su Heis big

\newcommand\dualR[2]{\langle #1,#2 \rangle_{\R^2}} %dualita su R^2

\newcommand\gel{\cG}  %trasformata di Gelfand

\newcommand\set[1]{\left\{#1\right\}}

\newcommand \inv{^{-1}}

\newcommand\haar{m}
\newcommand\suppfan{\rho}

\newcommand\iper{ {_{1} F}_{\!1}}

\def\inutile#1{{}}
\begin{document}

\title[Paley--Wiener Theorems on the Heisenberg group]
{Paley--Wiener Theorems \\for the $\Un$--spherical transform \\on the Heisenberg group}

\author[F. Astengo, B. Di Blasio, F. Ricci]
{Francesca Astengo, Bianca Di Blasio, Fulvio Ricci}

\address
{Dipartimento di Matematica\\
Via Dodecaneso 35\\
16146 Genova\\ Italy} \email{astengo@dima.unige.it}

\address
{Dipartimento di Matematica e Applicazioni\\
Via Cozzi 53\\
  20125 Milano\\ Italy}
\email{bianca.diblasio@unimib.it}

\address
{Scuola Normale Superiore\\
Piazza dei Cavalieri 7\\
56126 Pisa\\ Italy}
\email{fricci@sns.it}

\thanks{Work partially supported
by  MIUR (project ``Analisi armonica").
}



\begin{abstract}
 We prove several Paley--Wiener-type theorems related to the spherical transform
on the Gelfand pair $\big(\Hn\rtimes\Un,\Un\big)$, where  $\Hn$ is the 
$2\dimH+1$-dimensional Heisenberg group. 

Adopting the standard realization of the Gelfand spectrum as the Heisenberg fan in $\R^2$, 
we prove that spherical transforms of $\Un$--invariant functions and distributions 
with compact support in $\Hn$ admit unique entire extensions to $\C^2$, 
and we find  real-variable characterizations of such transforms.
Next, we characterize the inverse spherical transforms of compactly supported functions 
and distributions on the fan, giving analogous characterizations. 
\end{abstract}

\maketitle

\section{Introduction}

The spherical transform for the Gelfand pair $\big(\Hn\rtimes\Un,\Un\big)$ maps $\Un$--invariant
functions, i.e. radial functions,
on the Heisenberg group $\Hn$ to functions on the Heisenberg fan $\Sigma$, 
which is naturally realized as a closed subset of~$\R^2$, the {\it Heisenberg fan} defined in \eqref{cfan}. 
In \cite{ADR,ADR1} we have studied the image of the space $\schwkrad$ of radial Schwartz functions, 
showing that it consists of the restrictions to $\Sigma$ of Schwartz functions on $\R^2$.

In this paper we first use this result to extend the notion of spherical transform 
to tempered radial distributions, identifying such transforms 
as the distributions on $\R^2$ which are ``synthetizable''   on $\Sigma$, i.e., 
vanish on functions which are  identically zero on the fan.
Then we prove Paley--Wiener type theorems
for the spherical transform $\gel $ and its inverse.

The natural starting point   for establishing Paley-Wiener theorems for $\cG$ is the fact that, when $f$ has compact support, its spherical transform $\cG f$ can be extended
from the set of bounded spherical functions (the Gelfand spectrum) to the set of all spherical 
functions. Spherical functions are parametrized by the pairs $(\xi,\lambda)\in\C^2$ of their 
eigenvalues with respect to the two fundamental differential operators, $L$ (the sublaplacian) and 
$i\inv T$ (the central derivative).    Moreover, the
spherical function $\Phi_{\xi,\lambda}$ with eigenvalues $(\xi,\lambda)\in\C^2$ depends holomorphically on $(\xi,\lambda)$.
This allows to extend the  spherical transform of a function or distribution
with compact support to an entire function on $\C^2$.

Symmetrically, each spherical function $\Phi_{\xi,\lambda}$ extends to an entire function 
on the complexification  $ H_n^\C$ of $\Hn$, and  the inversion formula shows that 
if   $\cG f$  has compact support in the Gelfand spectrum, then 
the function itself extends to an entire function   on $H_n^\C$.

It does not look plausible to have a simple ``complex variable'' description of the entire functions which are in the range of the spherical, or inverse spherical, transform of the space of $C^\infty$-functions, or of distributions, with compact support, see also the comments in Fuhr~\cite{MN}, in a context that is  closely related to ours.

We rather look for analogues of the ``real variable'' characterization of the classical Paley-Wiener spaces in $\R^n$, in the spirit of the works of Bang~\cite{Bang} 
and Tuan~\cite{Tuan}, later expanded and refined by Andersen and deJeu~\cite{Nils}. 
We take the following as our model statement~\cite{Nils}:
 a function $f$ on $\R^n$ is the Fourier transform of a $C^\infty$ function with compact support if and only if it is a Schwartz function and, for some $p\in[1,\infty]$,
\begin{equation}\label{Delta}
 \limsup_{k\to\infty}\|\Delta^kf\|_p^\frac1k<\infty\ .
 \end{equation}
 
 In this case, the left-hand side is finite for every $p$, the ``$\limsup$'' is a limit and, for every $p\in[1,\infty]$,
 $$
 \lim_{k\to\infty}\|\Delta^kf\|_p^\frac1k=\max_{x\in{\rm supp}\cF^{-1} f}|x|^2\ .
 $$ 

When restricted to radial functions, this theorem can be reformulated in terms of the spherical transform $\cG$ for the Gelfand pair $(\R^n\rtimes{\rm SO}_n,{\rm SO}_n)$, given by $\cG f(\la)=\hat f(\xi)$ with $|\xi|^2=\la\ge0$. Then we have    two different statements, depending on the side of the Fourier transform it is applied on:
\begin{enumerate}
\item[(i)] a function $g$ on the Gelfand spectrum $[0,+\infty)$ is the spherical transform of a radial $C^\infty$ function on $\R^n$ with compact support if and only if it is a Schwartz function and, for some $p\in[1,\infty]$,
$$
 \limsup_{k\to\infty}\|g^{(k)}\|_p^\frac1k<\infty\ ;
 $$
  in this case, for every $p\in[1,\infty]$, 
$$
 \lim_{k\to\infty}\|g^{(k)}\|_p^\frac1k=\max_{x\in\supp \cG\inv g}|x|^2\ .
 $$
\item[(ii)] a radial function $f$ on $\R^n$ is the inverse spherical transform of a $C^\infty$ function with compact support in $[0,+\infty)$ if and only if it is a Schwartz function and \eqref{Delta} holds for some $p\in[1,\infty]$; in this case, for every $p\in[1,\infty]$,
 $$
 \lim_{k\to\infty}\|\Delta^kf\|_p^\frac1k=\max_{\la\in{\rm supp}\cG f}\la\ .
 $$
 \end{enumerate}
 
 We regard (i) as a Paley-Wiener theorem for the (direct) spherical transform, and (ii) as a Paley-Wiener theorem for the inverse spherical transform.
 
Possible  analogues of (i) and (ii) for the pair $\big(\Hn\rtimes\Un,\Un\big)$ rely on the identification, for each direction, of  a differential operator on one side of the spherical transforms and the corresponding ``norm'' on the other side. 

Our results are related to the following choices:
\begin{enumerate}
\item[(i')]  the difference/differential  operators $M_\pm$ of Benson, Jenkins and Ratcliff \cite{BJR} on $\Sigma$ and the Kor\'anyi norm   \eqref{norm} on $H_n$;
\item[(ii')] the sublaplacian on $H_n$ and its eigenvalue $\xi$ on~$\Sigma$.
\end{enumerate}

We first prove real Paley-Wiener theorems for the direct spherical transform, i.e. analogues of (i) with the ingredients in (i'). We
 treat the cases of $C^\infty$ and $L^2$ functions and of  tempered distributions. These characterizations are summarized in Therorem~\ref{unoequiv},  Corollary~\ref{L2} and Theorem~\ref{sch}.   
 
 We also remark that 
the (unique) entire extension of the transform of  a function in ${\mathcal D}_{\text{\rm rad}}(\Hn)$
needs not be Schwartz on $\R^2$. This shows that, in general, the Schwartz extensions 
to $\R^2$ constructed in \cite{ADR1} are  different from  the entire extension discussed here.

In the second part of the paper, we show that, given a  distribution~$U$ on $\R^2$ with compact support , the inversion formula for the spherical transform produces a function on the Heisenberg group $\Hn\simeq \R^{2\dimH+1}$ which can be analytically extended
 to $\C^{2\dimH+1}$. 
 If the distribution $U$ is synthetizable on $\Sigma$, the function so obtained on $H_n$ coincides with its inverse spherical transform. For such distributions $U$, we obtain a real Paley-Wiener analogue of (ii) with the ingredients in~(ii').  A similar theorem is also proved for functions on $\Sigma$ which are either restrictions of $C^\infty$ functions or are square integrable with respect to the Plancherel measure.
 Our characterizations are summarized in Therorem~\ref{maininv}, Theorem~\ref{L2inv} and Theorem~\ref{schinv}.
 These results can be interpreted as a ``real''  spectral Paley-Wiener theorems 
 for the spectral measure of the sublaplacian, a point of view which coincides  with that of \cite{MN}.

There is a wide literature on Paley--Wiener theorems on the Heisenberg group. 
The earliest result is due to Ando~\cite{PJapAc},
followed by
 Thangavelu~\cite{Revista,JFA,HirMJ}, 
Arnal and Ludwig \cite{PAMS},
 Narayanan and Thangavelu \cite{AnnIFour}.
  Results are mostly  related to the group (operator-valued) Fourier transform 
 and its inverse, but there are also ``spectral'' Paley--Wiener theorems,
  as in the already mentioned paper \cite{MN}, where the condition of compact support 
  on the transform of a given function is replaced 
 by the condition that the function itself belongs to the image of the 
 spectral measure of a compact set in $\R^+$ 
 associated to the sublaplacian (see also Strichartz \cite{JFA-Laplacians},
  Bray \cite{Monats.M},
 and Dann and \'Olafsson~\cite{Olaf} in other contexts).
  
Our paper is organized as follows. In Section~2 we introduce the basic notation. In Section~3 we treat spherical functions
noting that they can be extended to holomorphic functions in each variable
and providing some easy estimates. Section~4 and Section~5 deal 
with the spherical transform of radial functions and radial tempered distributions, respectively. 
In Section 6 we prove some properties of  the operators $M_\pm$
first introduced in \cite{BJR}.
These are exploited in Sections 7 and 8 to obtain real Paley--Wiener theorems
for the spherical transform and its inverse, respectively.

\section{Notation}
We denote by $\Hn$ the Heisenberg group, i.e.,
the real manifold $\C^\dimH\times\R$ equipped with the group law
$$
(z,t)(w,u)=\bigl(z+w,t+u+\tfrac12\,\IM \langle w | z \rangle \bigr)
\qquad\forall z,w\in \C^\dimH,\quad
t,u\in \R,
$$
where $\langle\cdot|\cdot\rangle$ denotes the Hermitian innner product in $\C^\dimH$.

It is easy to check that the Lebesgue measure $dm=dz\,dt$ is a Haar measure on $\Hn$.
\medskip

We denote by $T$, $Z_j$ and $\bar{ Z_j}$, where $j=1,\ldots,\dimH$,
the left-invariant vector fields
$$
Z_j=\partial_{z_j}-\tfrac{i}{4}\, \bar{z}_j\,\partial_t
\qquad
\bar{Z}_j=\partial_{\bar{z}_j}+\tfrac{i}{4}\, z_j\,\partial_t ,
\qquad T=\partial_t\ .
$$
The only nontrivial brackets are $T=-2i\,[Z_j,\bar Z_j]$.

The operators $Z_j$ and $\bar Z_j$ are homogeneous of degree $1$
while $T$ is homogeneous of degree~$2$
with respect to the anisotropic dilations
$r\cdot (z,t)=(rz,r^2t)$, where $r>0$ and $(z,t)\in \Hn$.
Let $I=(i_1,\ldots,i_{2\dimH+1})$ be in $\N^{2\dimH+1}$;
we denote by ${\op D}^I$ a
differential operator of homogeneous degree 
$\text{deg}\,I=i_1+\cdots+i_{2\dimH}+2i_{2\dimH+1}$ of the form
\begin{equation}
\label{monomi}
{\op D}^I=Z_{1}^{i_1} \bar Z_1^{i_2}
\cdots Z_{\dimH}^{i_{2\dimH-1}} \bar Z_n^{i_{2\dimH}} T^{i_{2\dimH+1}} .
\end{equation}
The monomials ${\op D}^I$ with  $\text{deg}\,I=j$
form a basis of the space of all left-invariant differential operators on $\Hn$ 
which are homogeneous of degree
$j$.

We write $\cS (\Hn)$ for the Schwartz space of
functions on $\Hn$, i.e., the space of infinitely differentiable
functions~$f$ on~$\Hn$ such that
all  partial derivatives~${\op D}^I f$ of~$f$
are rapidly decreasing.
The Schwartz space is equipped with the following
family of norms, parametrized by a nonnegative integer $p$:
$$
\|f\|_{(p)}=\sup_{(z,t)\in \Hn}\{(1+| (z,t)|)^{p}\, |{\op D}^I f(z,t)|
\,:\, \text{deg}\, I\leq p\}\ ,
$$
where 
\begin{equation}\label{norm}
| (z,t)|=\left( \frac{|z|^4}{16}+t^2\right)^{1/4} .
\end{equation}
We also define $\cA$ as    
$$
\mathcal A(z,t)= \frac{|z|^2}{4}+it 
\qquad \forall (z,t)\in \Hn,
$$ 
so that 
$
|A(z,t)|=| (z,t)|^2.
$

\section{Spherical functions}
The unitary group $\Un$  
acts on $\Hn$ via
$$
k\cdot (z,t)=(kz, t) \qquad \forall (z,t)\in \Hn,\quad
k\in \Un.
$$ 
This action induces an action on functions~$f$ on $\Hn$ by the formula
$$
k\cdot f(z,t)=f(k^{-1}z,t)\qquad \forall k\in \Un,\quad (z,t)\in \Hn.
$$
We note that a function $f$ on $\Hn$ is $\Un$--invariant if and only if
it depends
only on $|z|$ and~$t$
, therefore we shall call it radial.
We denote by $\schwkrad$ the space
 of  radial
 Schwartz  functions.
  
  Denote by $G$ the   semidirect product $\Hn\rtimes \Un$.
    We may identify the space of smooth bi-$\Un$--invariant functions  $\cD(G/\!/\Un)$ 
with the   algebra  $\crad(\Hn)$ 
of smooth radial  functions on $\Hn$ with compact support. It is known~\cite{HR,DR} that 
$(G,\Un)$ is a Gelfand pair, i.e.,
$\crad(\Hn)$  is a commutative algebra.  
We may also identify the commutative 
algebra $\mathbb D(G/\Un)$ of $G$--invariant differential operators on 
$G/\Un$   with the algebra
$\algop$ of all left-invariant and
$\Un$--invariant differential operators on $\Hn$, which    has
two essentially self-adjoint 
generators, namely $i^{-1}T$
and the 
  sublaplacian
$$
L=-2\sum_{j=1}^\dimH\bigl( Z_j\bar Z_j+\bar Z_j Z_j\bigr) .
$$

The   spherical 
functions are characterized as the  
joint eigenfunctions
of all 
$G$--invariant differential operators on 
$G/\Un$, i.e., as the radial eigenfunctions of $i^{-1}T$
and $L$,  normalized to take value $1$ at identity.
Spherical functions are analytic and are uniquely determined by the pair  of their eigenvalues relative to 
$L$ and $i^{-1}T$ respectively. 

The next subsection  shows that the spherical function 
$\Phi_{\xi,\la}$ exists for every pair $(\xi,\la)$ of 
eigenvalues  and that  depends  holomorphically
on variables and parameters.

\subsection{Holomorphy of spherical functions}

We initially consider real eigenvalues $\xi$ and $\la$ and look for a radial solution of the system
\begin{equation}\label{sistema}
\left\{
\begin{array}{l}
Lu=\xi\, u
\\
Tu=i\la\, u
\\
u(0,0)=1\ .
\end{array}
\right.
\end{equation}
Following~\cite{Kora}, for $\la\neq 0$  we write the solution in the form 
$$
u(z,t)=e^{i\la t}\,e^{-\la |z|^2/4}\,v(\la |z|^2/2),
$$
obtaining that $v$ satisfies
the confluent hypergeometric  differential equation
\begin{equation}\label{eqdiffconfl}
s\,v''(s)+(c-s)v'(s)-a\,v(s)=0
\end{equation}
with parameters $a=\frac{\dimH}{2}-\frac{\xi}{2\la}$, $c=\dimH$.
The normalized solution of \pref{eqdiffconfl} is the confluent hypergeometric function
$$
\iper(a,c;s)=1+\frac{a}{c}\, s+\frac{a(a+1)}{c(c+1)}\, \frac{s^2}{2!}+\cdots
=\sum_{k=0}^\infty \frac{(a)_k}{(c)_k}\, \frac{s^k}{k!}\ ,
$$
where $(a)_0=1$, $(a)_k=\Gamma(a+k)/\Gamma(a)$,
so that for real $\la\neq 0$
$$
u(z,t)=e^{i\la t}\,e^{-\la |z|^2/4}\,
\iper(\tfrac{\dimH}{2}-\tfrac{\xi}{2\la},\dimH;\la |z|^2/2)\ .
$$
When $\la=0$  and $\xi$ is real,
a similar procedure shows that 
$$
u 
(z,t)=\sum_{k=0}^\infty \frac{(-1)^k}{(\dimH)_k}\,
\frac{(\xi |z|^2/4)^k}{k!}={\mathcal J}_{\dimH-1}(\xi |z|^2/4)
\qquad\forall (z,t)\in \Hn,
$$
where 
$$
{\mathcal J}_{\beta}(s) 
= 
\sum_{k=0}^{+\infty}
\frac{(-s)^k }{k!\,(\beta+1)_k}
\qquad \forall s\in \C.
$$
Note that $J_{\beta}(u)=\displaystyle\frac{(u /2)^{\beta }}{\beta!\,\,}  
{\mathcal J}_{\beta}\left(u^2 /4\right)$
  is the   Bessel function of the first kind of order $\beta$.
  
  Therefore for every pair of  real 
numbers $\xi$ and $\la$ we have the spherical function
\begin{equation*} 
  \Phi_{\xi,\la}(z,t)=
\begin{cases}
 e^{i\la t}\, e^{-\la |z|^2/4}\, 
\iper \left(\frac{\dimH}{2} -\frac{\xi}{2 \la},\dimH;\frac{\la |z|^2}{2}\right) 
&  
\la\not= 0
\\
{\mathcal J}_{\dimH-1}(\xi |z|^2/4)&    
\la = 0
\end{cases}
\qquad \forall (z,t)\in \Hn.
\end{equation*}
 We now verify
  that  $\la\longmapsto  \Phi_{\xi,\la}(z,t)$ is regular in $\la=0$.
Indeed, something more holds.
 
\begin{lemma}\label{olosferiche}
The function  $(x,y,t,\xi,\la)\longmapsto   
\Phi_{\xi,\la}(x+iy,t)$ extends to a holomorphic function on 
$\C^{2\dimH+3}$.
\end{lemma}

\begin{proof}
Note that when $\la\not=0$, $z=x+iy$,
$$
\begin{aligned}
\iper \left(\frac{\dimH \la-\xi}{2 \la},\dimH;\frac{\la |z|^2}{2}\right)
&= 
\sum_{k=0}^{\infty}
\frac{\left( 
\frac{\dimH \la-\xi}{2 \la}
\right)_k}
{(n)_k\,k! }
   \left(
 \frac{\la |z|^2}{2}
 \right)^k 
 \\
  &=1+
\sum_{k=1}^{\infty}
\frac{(x^2+y^2)^{k}}{(n)_k\,k!\,4^k}\,\,
\prod_{d=0}^{k-1}\left(
  \la(2d+\dimH)-\xi  \right)
,
  \end{aligned}
$$ 
so that for all $(\xi,\la,x,y,t)$ the function
\begin{equation}\label{Phi}
\Phi_{\xi,\la}(x+iy,t)
=e^{i\la t}\, e^{-\la (x^2+y^2)/4}\, \left(1+
\sum_{k=1}^{\infty}
\frac{(x^2+y^2)^{k}}{(n)_k\,k!\,4^k}\,\,
\prod_{d=0}^{k-1}\left(
  \la(2d+\dimH)-\xi  \right)\right)
\end{equation}
is a series of entire functions converging  uniformly
on compact sets.
\end{proof}

By analytic continuation, $\Phi_{\xi,\la}$ 
is the spherical function for every $(\xi,\la)\in\C^2$.

\subsection{Bounded spherical functions} 
 The Gelfand spectrum of the Banach algebra $\Lunorad$ of radial integrable functions  
is given by the set 
  of normalized bounded spherical functions, equipped with the compact open topology.   
 We recall
  \cite{Kora}  that  $\Phi_{\xi,\la}$ is bounded on $\Hn$ 
  if and only if $(\xi,\la)$ 
belongs to the so called  Heisenberg fan  given by
\begin{equation}\label{cfan}
\cfan=\fan\cup \set{(\xi,0)\in \R^2\,\,:\,\, \xi\geq 0},
\end{equation}
where
$$
\fan=\set{(\xi,\la)\in\R^2:\la\ne0\,,\, \xi=|\la|(2j+ \dimH)\,,\, \,  j\in \N }\ .
$$
It is known that $\cfan$ is homeomorphic to  the Gelfand spectrum~\cite{FR, BJRW}. 
\inutile{
Una maniera potrebbe essere quella di tagliare del tutto e scrivere: stime standard
sulle confluenti implicano che $\Phi_{\xi,\la}$ \`e limitata su $\Hn$
se e solo se $(\xi,\la)$ 
belongs to the so called  Heisenberg fan  given by ...
Moreover by (la formula lebedev) ... and
when $(\la,|\la|(2j+\dimH))$ is in $\fan$, 
polinomi di Laguerre, which is the form that we usually find in literature.
}
 
When $(\xi,\la)$ is in $\fan$, the spherical function
$\Phi_{\xi,\la}$ can be written 
in terms of Laguerre polynomials, which is the form that we usually find in literature.
Indeed,  
the relation (see \cite[p.~253, formula~(7)]{E}
\begin{equation}\label{lebedev}
\iper(a,\dimH;x)=e^x\, \iper(\dimH-a, \dimH;-x),
 \end{equation} 
implies that 
$$
\Phi_{\xi,\la}(z,0)
=\Phi_{\xi,|\la|}(z,0)=
e^{-|\la| |z|^2/4}\,
\iper \left(\frac{\dimH}{2}-\frac{\xi}{2 |\la|},\dimH;\frac{|\la| |z|^2}{2}\right)
$$
and when $\xi=|\la|(2j+\dimH)$ 
the hypergeometric function in the previous formula coincides with 
the normalized
   $j^{\text{th}}$ Laguerre polynomial
 of order $\dimH -1$, i.e., 
 $$
 \iper \left(-j,\dimH;\frac{|\la| |z|^2}{2}\right) 
 =\frac1{\binom{j+ \dimH-1}{j}}\sum_{k=0}^j \binom{j+\beta }{ j-k}
   \frac{(-\frac{|\la| |z|^2}{2})^k}{ k!}.
 $$
\medskip

\subsection{Estimates of derivatives of spherical functions}
In this subsection we exploit the fact that  the bounded  spherical functions 
$\left\{\Phi_{\xi,\la}\right\}_{(\xi,\la)\in \cfan}$  are averages of coefficients of irreducible 
unitary representations of $\Hn$   
to give some estimates that we shall need in the sequel.
Referring to the Bargmann-Fock model 
of the irreducible representations $\pi^\lambda$ of $H_n$ 
associated to the character $e^{i\lambda t}$ on the center, 
we represent the operators $\pi^\lambda(z,t)$ as matrices 
$\big(\pi^\lambda_{\bf j,\bf k}(z,t)\big)_{\bf j, \bf k\in \N^\dimH}$ 
in the basis of normalized monomials (for more details see the monographs~\cite{Folland}
or~\cite{Th}). Then the bounded spherical functions 
can be written as averages of
diagonal entries of this matrix   according to the rule
\begin{equation}\label{defPhi}
\Phi_{\xi,\la}=\frac1{\binom{j+ \dimH-1}{j}}\sum_{|{\bf j}|=j}\pi^\lambda_{\bf j,\bf j},
\end{equation}
 where 
$\xi=|\la|(2j+\dimH)$ and $|{\bf j}|={\bf j}_1+\ldots +{\bf j}_n$ for $ {\bf j}  =\left({\bf j}_1, \ldots ,{\bf j}_n\right)\in \N^\dimH$.

\begin{lemma}\label{stimasferiche} Let ${\op D}^I$ be a
differential operator of homogeneous degree $\deg I=  \al$ as in\pref{monomi}. Then  
$$
|{\op D}^I \Phi_{\xi,\la}(z,t)|\leq C_\al\, (1+\xi)^{\al /2}
\qquad\forall (\xi,\la)\in \cfan,\,\, (z,t)\in \Hn.
$$
 \end{lemma}

\begin{proof} Let $\la\neq 0$.
Here and afterwards, if any component of the multiindeces ${\bf j}$ or ${\bf k}$
is negative, then $\pi^\lambda_{\bf j,\bf k}=0$.
Since the representations are unitary, $|\pi^\lambda_{\bf j,\bf k}|\leq 1$ and
it is easy to check that
$$
Z_i \pi^\lambda_{\bf j,\bf k}=
\begin{cases}
-\sqrt{\frac{k_i\la}{2}}\, \pi^\la_{{\bf j},{\bf k}-{\bf e}_i}
&\la>0
\\
\sqrt{\frac{(k_i+1)|\la|}{ 2}}\, \pi^\la_{{\bf j},{\bf k}+{\bf e}_i}
&\la<0
\end{cases}
\qquad
\bar Z_i \pi^\la_{{\bf j},{\bf k}}=
\begin{cases}\sqrt{\frac{(k_i+1)\la}{2}}\, \pi^\la_{{\bf j},{\bf k}+{\bf e}_i}
&\la>0
\\
-\sqrt{\frac{k_i|\la|}{ 2}}\,\pi^\la_{{\bf j},{\bf k}-{\bf e}_i}
&\la<0\ 
\end{cases}
$$
and $T\pi^\la_{{\bf j},{\bf j}}=i\,\la\, \pi^\la_{{\bf j},{\bf j}}$.
Here ${\bf e}_i$ is the multiindex with just the $i^{th}$ component equal to $1$.

Suppose that $(\xi,\la)$ is in $\fan$, with $\xi=|\la|(2|{\bf j}| +\dimH)$.
Then if ${\op D}^I =Z^{\bf k}\bar Z^{\bf h} T^s$ with 
$\deg I=\al =|{\bf k}|+|{\bf h}|+2s$ we have 
\begin{align*}
|
{\op D}^I\pi^\la_{{\bf j},{\bf j}}
|
&
=|\la|^s\,
|Z^{\bf k}\bar Z^{\bf h}\pi^\la_{{\bf j},{\bf j}} |
\\
&=|\la|^s\,
\begin{cases}
\sqrt{\prod_{i=1}^n\prod_{\ell=1}^{h_i}\tfrac{\la}{2}(j_i+\ell)}\, 
\sqrt{\prod_{i=1}^n\prod_{\ell=1}^{k_i}\tfrac{\la}{2}(j_i+h_i+1-\ell)}\,
|\pi^\la_{{\bf j},{\bf j+h-k}}|
&\la>0
\\
\sqrt{\prod_{i=1}^n\prod_{\ell=1}^{h_i}\tfrac{|\la|}{2}(j_i+1-\ell)}\,
\sqrt{\prod_{i=1}^n\prod_{\ell=1}^{k_i}\tfrac{|\la|}{2}(j_i-h_i+\ell)}\,
|\pi^\la_{{\bf j},{\bf j-h+k}}|
&\la<0
\end{cases}
\\
&\leq C\,
|\la|^s\,\sqrt{
 |\la|^{|{\bf h}|+|{\bf k}|}(|{\bf j}|+|{\bf h}|)^{|{\bf h}|}
\,\,\,
(|{\bf j}|+|{\bf h}|+|{\bf k}|)^{|{\bf k}|}}
\\
&\leq C_{\al}\, (1+\xi)^{\al/2}.
\end{align*}
Here we have used the fact that in $\fan$ we have $\xi=|\la|(2|{\bf j}|+\dimH)\geq |\la|$.
By~\pref{defPhi} the thesis follows on $\fan$, and 
by continuity the same estimates hold on $\cfan$, thus proving the lemma. 
\end{proof}

\section{Spherical transform}
As usual, we denote by  $\dualH{\cdot}{\cdot}$ the dual  pairing on 
 $\Hn$ and we shall also write 
$$
\dualH{f}{g}=\int_{\Hn} f(z,t)\, g(z,t)\, dz\, dt,
  \qquad \forall f,g\in \cS(\Hn).
$$
Given a measurable  function $f$ on $\Hn$  we  
denote by $\check f$   the function defined by 
 $\check f(x)=f(x^{-1})$ for every $x$ in $\Hn$.

\subsection{Definitions and main facts}
\label{defgel}
Let $f$ be in $\Lunorad$. We  define   its   spherical 
transform~$\gel f$     by  
$$
\gel f(\xi,\lambda) =\int_{\Hn} f(x)\, \Phi_{\xi,\lambda}(x^{-1})\, dx=
\dualH{f}{\check\Phi_{\xi,\lambda}}
\qquad \forall (\xi,\la)\in \cfan.
$$
Then $\gel f$ is a continuous function on $\cfan$. 

 The inversion formula for a function~$f$ in $\schwkrad$ is 
$$
f(x)=\int_\cfan \gel f(\xi,\la)\, \Phi_{\xi,\la}(x)\, d\mu(\xi,\la)
\qquad \forall x\in \Hn,
$$
where $\mu$ is the Plancherel measure defined by  
$$
\int_\cfan \psi \, d\mu
=\frac{1}{(2\pi)^{\dimH+1}}\,\int_\R 
  \sum_{j=0}^{\infty}\binom{j+ \dimH-1}j\,
  \psi(|\la|(2j+ \dimH),\la)\, |\la|^{\dimH}\, d\la
\qquad\forall \psi\in  
C_c(\cfan).
$$

It is easy to check that the function $(\xi,\la)\mapsto (1+\xi)^{-(\dimH+2)}$
is in $L^1(\cfan)$, so  
\begin{equation}
\label{integrabilita}
\| \psi \|_{L^1(\cfan)}\leq C\,
\|(1+|\xi|)^{\dimH+2}\,\psi \|_{L^\infty(\R^2)}
\qquad \forall \psi\in \cS(\R^2).
\end{equation}

As in~\cite{ADR}, we denote by 
$\cS(\cfan)$  the space of 
restrictions to $\cfan$ of Schwartz functions on $\R^2$,
endowed with the quotient topology $\cS(\R^2)/\set{\phi\, :\, \phi|_\cfan=0}$.
For  radial Schwartz functions  on $\Hn$, we have proved
the following result.
 
\begin{theorem}\textnormal{\cite[Corollary 1.2]{ADR}} \label{nostro}
The spherical transform is a topological isomorphism between
the spaces $\schwkrad$
 and $\cS(\cfan)$.
\end{theorem}

On the other hand, Lemma~\ref{olosferiche}
implies that when $f$ is compactly supported
we can regard $\gel f$ as a function on $\C^2$.

\begin{proposition}  \label{olomorfia}
If  $f$ is  in $\crad(\Hn)$ then  $\gel f$ extends 
to the  holomorphic function $F$ 
on $\C^2$ given by the rule
$$
F(\xi,\la)=\dualH{f}{\check\Phi_{\xi,\la}}
\qquad\forall (\xi,\la)\in \C^2.
$$
\end{proposition}

\subsection{Holomorphic versus Schwartz extensions}\label{sec:olo}
Given $f$ in $\crad(\Hn)$, we have found two
ways of extending its spherical transform $\gel f$ to a smooth 
function
on $\R^2$. Namely, by Theorem~\ref{nostro}, there exists a Schwartz function
$G$ on $\R^2$ such that $G|_\cfan=\gel f$, and by Proposition~\ref{olomorfia}
the  function $F$ is the holomorphic extension of $\gel f$ to $\C^2$.
So $G|_\cfan=F|_\cfan=\gel f$.

We observe that any two entire functions on $\C^2$,
which coincide on~$\cfan$, 
are everywhere equal, so $F$ is the unique
continuation of $\gel f$
to an entire function on $\C^2$. 

A question arises naturally: 
if $f$ is  in $\crad(\Hn)$, is it true that $F$, 
when restricted to real values of $(\xi,\la)$,  
is a Schwartz function on $\R^2$?

In the rest of this subsection we show that the answer   can be negative.

Let $f$ be a function of the form 
$$
f(z,t)=g(z)\otimes h(t) \qquad \forall (z,t)\in \Hn
$$ 
where $h$ is even and
compactly supported
and $g$ is nonpositive and supported in  $1<|z|<4$,  equal to $-1$
when $2<|z|<3$.

Let $\mathcal{F}h$ be the Euclidean Fourier transform of $h$.
We now show that the function $\la\ge 0\mapsto |\mathcal{F}h(\la)|\,e^{\la/2}$
is not bounded. Indeed, since $h$ is even, if it were bounded, 
then  
for every $b$, $0\leq b<1/2$, the function 
$\la\mapsto e^{b|\la|}\mathcal{F}h(\la)$ would be in $L^2(\R)$. By 
the Paley--Wiener Theorem for the Euclidean transform 
$h=\mathcal{F}^2h$ would continue analytically to
 $\set{w\,:\, |\IM(w)|<1/2}$, but this cannot
 be true since $h$ has compact support.
 Therefore the function $\la\ge 0\mapsto |\mathcal{F}h(\la)|\,e^{\la/2}$
is not bounded. 

Now, if $(\xi,\la)\in \R^2\mapsto F(\xi,\la)$ were rapidly decreasing, then the same
would hold true for the function
$\la\mapsto F((\dimH+1)\la,\la)$. Note that when $\la>0$
$$
F((\dimH+1)\la,\la)=\mathcal{F}h(\la) 
\int_{1<|z|<4}g(z)\, e^{-\la |z|^2/4}\, \iper(-\tfrac{1}{2}, \dimH,\la |z|^2/2)\, dz.
$$

Moreover $\iper(-\tfrac{1}{2},\dimH,x)\leq 0$ when $x\geq 2\dimH$
 and by the estimate 
(see~\cite[p.~27, formula~(3)]{E})
$$\iper(a,\dimH;x)=\frac{\Gamma(\dimH)}{\Gamma(a)}e^x\,x^{a-\dimH}(1+O(x^{-1})),
\qquad \RE(x)\to\infty,\quad a\not=0,-1,-2,\ldots,
$$
when $\la\to+\infty$ we obtain 
\begin{align*}
|F((\dimH+1)\la,\la)|
&=|\mathcal{F}h(\la)|\, 
\int_{1<|z|<4}  g(z)\,e^{-\la |z|^2/4}\, \iper(-\tfrac{1}{2}, \dimH,\la |z|^2/2)\, dz
\\
&
\geq C\,
|\mathcal{F}h(\la)|\, \int_{2<|z|<3}  e^{-\la |z|^2/4}\,
e^{\la |z|^2/2}\, (\la |z|^2/2)^{-1/2-\dimH}\,\, dz
\\
&\geq C\, |\mathcal{F}h(\la)|\,e^{\la/2},
\end{align*}
so the function $\la\mapsto F((\dimH+1)\la,\la)$ is not bounded.

%%%%%%%%%%%%%%%%%%%%%

\section{The spherical transform of radial tempered distributions}
   
As usual, we denote by $\dualR{\cdot}{\cdot}$   the dual  pairing on 
$\R^2$   and we also write 
$$
\dualR{\varphi}{\psi}=\int_{\R^2} \varphi(x)\,\psi(x)\, dx\qquad \qquad
\forall \varphi,\psi\in \cS({\R^2}).
$$  
Let $\Pi\, :\, \cS(\Hn) \longrightarrow \cS(\Hn) $ be the averaging projector  
defined by  
$$
\Pi f 
=\int _{\Un}f\circ k\,\,dk\qquad \forall f\in \cS(\Hn).
$$
% 
%%%%%%%%%
Then the Schwartz space on $\Hn$ decomposes into the direct sum 
\hbox{$\cS(\Hn)= \schwkrad \oplus \ker \Pi$} so that  
$ \schwkrad $  is isomorphic to the quotient space 
 $\cS(\Hn)/ \ker \Pi$. It follows that  
 the dual space  $ \big(\schwkrad\big)' $ is isomorphic    
 to the subspace $\schwkradp$ of $\cS'(\Hn)$
 consisting of all 
tempered distributions $\Lambda$  on $\Hn$ such that
$$ 
\dualH{ \Lambda}{ f}=0\ \qquad\forall\,f\in\ker \Pi.
$$
 On the other hand the dual space of $\cS(\cfan)$ 
 is naturally isomorphic   to the subspace $\cS_0'(\cfan)$ of  $\cS'(\R^2)$ 
 consisting of all 
tempered distributions $U$  on $\R^2$ such that
\begin{equation*} 
\dualR{ U}{ g}=0\ ,\qquad\forall\,g\in\cS(\R^2)\text{ such that }g=0\text{ on }\cfan\ .
\end{equation*}
 We note that the Plancherel formula can be written as
$$
\dualH{f}{\overline g}= \dualR{\gel f\, \mu}{\overline{\gel g}}
= \dualR{\gel f\, \mu}{{\gel \overline{\check g}}}
\qquad \forall f,g\in \schwkrad.
$$  
Therefore we are led to define
the spherical transform of a radial  tempered distribution $\Lambda$ on $\Hn$ as the 
distribution $\gel \Lambda$ in $\cS'(\R^2)$ given  by 
\begin{equation}\label{tildeT}
\dualR{\gel \Lambda}{\vp}=\dualH{\Lambda}{(\gel\inv \vp_{|_\Sigma})^{\check{\phantom a}}}
\qquad \forall \vp\in \cS(\R^2). 
\end{equation}
Clearly, if $\Lambda$ is in $\schwkradp$, then for every 
function $\vp$ in $\cS(\R^2)$ such that $\vp=0$  on $\cfan$ we have
$
\dualR{\gel \Lambda}{\vp}
=\dualH{\Lambda}{(\gel\inv \vp_{|_\Sigma})^{\check{\phantom a}}}
=0,
$
i.e., $\gel\Lambda$ is in $\cS_0'(\cfan)$.  
We recall that we have denoted by $\haar$ the Lebesgue measure on $\Hn$.
If  $f$ is a radial function in $L^1\cap L^2(\Hn)$,
then $f\haar$ is in $\schwkradp$
and $\gel (f\haar)=(\gel f)\,\mu$ where $\gel f$ has been defined in Subsection~\ref{defgel},
so formula~\pref{tildeT} provides an extension of the usual spherical transform.
 
%%%%%%%%%%%%%%%%%%%%%%%% 

Moreover it is easy to verify 
that $\gel{(L^j \Lambda)} 
=\xi^{j}\, \gel \Lambda 
$
for every $\Lambda $ in $\schwkradp$.
\inutile{ Indeed,
$$
\dualR{\gel \Lambda}{\vp}
=\dualH{f}{(\gel\inv \vp_{|_\Sigma})^{\check{\phantom a}}}
=\dualR{\gel f\,\mu}{\gel (\gel\inv \vp_{|_\Sigma})}
=\dualR{\gel f\,\mu}{\vp}
\qquad\forall \vp \in \cS(\R^2).
$$
}
%%%%%%%%%%%%%%%%%%%%%%%% 

With this notation Theorem~\ref{nostro} extends to radial tempered distributions 
in the following form.

\begin{corollary}\label{nostrodist} The spherical transform 
$\cG$  is a topological isomorphism between
the spaces $\schwkradp$
 and $\cS_0'(\cfan)$.
\end{corollary}
 
 We now study the behavior of the spherical transform of 
  radial compactly supported distributions.    
\begin{proposition}\label{distolo} 
Let $\Lambda$ be a radial compactly supported distribution on $\Hn$. Then   
\begin{equation}
\widehat \Lambda\; :\; (\xi,\lambda)  \longmapsto \dualH{\Lambda}{\check\Phi_{\xi,\lambda}}
\end{equation}
   is a holomorphic function on $\C^2$. Moreover
 $\widehat \Lambda\,\mu$ is in $\cS_0'(\cfan)$ and
$$
\gel \Lambda= \widehat \Lambda\,\mu\ ,
$$
i.e., $\gel \Lambda$ coincides with the function $\widehat \Lambda$.
\end{proposition}

\begin{proof} Using Lemma~\ref{olosferiche}
it is easy to prove that $\widehat \Lambda$ is entire.
For the second part, we first check that 
for every $\psi$ in $\cS(\R^2)$, the integral $\int_\cfan \widehat \Lambda\, \psi\,d\mu$ 
is absolutely convergent, and therefore $\widehat \Lambda\,\mu$ is in $\cS_0'(\cfan)$.
Indeed, if $(\xi,\la)=\big(|\la|(2j+\dimH),\la\big)$ is in~$\fan$, for some $m$ in $\N$
$$
|\widehat \Lambda(\xi,\la)|=
|\dualH{ \Lambda}{\check{\Phi}_{\xi,\la}}| \le C\,\big\|
{\check\Phi_{\xi,\la}}\,\big\|_{C^m(K)}\ ,
$$
where $K=\supp \Lambda \subset \cfan$.
By Lemma~\ref{stimasferiche}, the function
 $\widehat \Lambda$ is slowly growing on $\Sigma$
and so for every $\psi $ in $\cS(\R^2)$, the integral $\int_\cfan \widehat \Lambda\, \psi\,d\mu$ 
is absolutely convergent.

When $g$ is in $\cS(\R^2)$,
\begin{align*}
\dualR{ \widehat \Lambda\,\mu}{\psi}
&=\int_\cfan \widehat \Lambda(\xi,\la)\, \psi(\xi,\la)\,d\mu\\
&=\frac{1}{(2\pi)^{\dimH+1}}\,
\int_\R\sum_{j=0}^\infty\binom{j+ \dimH-1}j\,\widehat \Lambda\big(|\la|(2j+\dimH),\la\big)
\, \psi\big(|\la|(2j+\dimH),\la\big)\,|\la|^\dimH\,d\la\\
&=\lim_{N\to\infty} \frac{1}{(2\pi)^{\dimH+1}}\,
\int_{-N}^N
\sum_{j=0}^N \binom{j+ \dimH-1}j\, \widehat \Lambda\big(|\la|(2j+\dimH),\la\big)\,
\psi\big(|\la|(2j+\dimH),\la\big)\,|\la|^\dimH\,d\la\\
&=\lim_{N\to\infty}\frac{1}{(2\pi)^{\dimH+1}}\,
\int_{-N}^N\sum_{j=0}^N  \binom{j+ \dimH-1}j\,
\dualH{ \Lambda}{
\check\Phi_{|\la|(2j+\dimH),\la}}\, \psi\big(|\la|(2j+\dimH),\la\big)\,|\la|^\dimH\,d\la\\
&=\lim_{N\to\infty}\frac{1}{(2\pi)^{\dimH+1}}\,
\, \bigdualH{ \Lambda}{\int_{-N}^N\sum_{j=0}^N \, \binom{j+ \dimH-1}j\,
            \check\Phi_{|\la|(2j+\dimH),\la}\,\, 
            \psi\big(|\la|(2j+\dimH),\la\big)\,|\la|^\dimH\,d\la}
\ .
\end{align*}

Since
$$
\lim_{N\to\infty}\frac{1}{(2\pi)^{\dimH+1}}\,
\int_{-N}^N\sum_{j=0}^N \, \check\Phi_{|\la|(2j+\dimH),\la}\, 
\psi\big(|\la|(2j+\dimH),\la\big)\,
|\la|^\dimH\,d\la=(\gel\inv  \psi_{|_\Sigma})\check{\phantom a}
$$
uniformly on compacta and the same holds for all derivatives,
\[
\dualR{ \widehat \Lambda\,\mu}{\psi}
=\dualH{ \Lambda}{(\cG\inv \psi_{|_\Sigma})\check{\phantom a}}
=\dualR{\gel \Lambda}{\psi}\ .\qedhere
\]
\end{proof}

\section{The operators ${M}_\pm$}

Denote by   ${M}_\pm$   the  operators acting on a
smooth function   $\psi$  on $\R^2$ by the rule~\cite{BJR}
 
$$
\begin{aligned}
  M_\pm \psi(\xi,\la)&= \partial_\la \psi(\xi,\la)\mp\dimH
   \partial_\xi \psi(\xi,\la)-2(\dimH\la\pm\xi)
\int_0^1\partial^2_\xi \psi(\xi\pm2\la t,\la )(1-t)\,dt. 
 \\
 &=\frac{1}{\la}\left(\la \partial_\la+\xi\partial_\xi\right)\psi(\xi,\la)-
 \frac{\dimH\la\pm\xi}{2\la^2}\big(\psi(\xi\pm 2\la,\la)-\psi(\xi,\la) \big).
 \end{aligned}
$$
Since $\la \partial_\la+\xi\partial_\xi$ is  the derivative in the radial 
direction,
the operators ${M}_\pm$ depend only on the restriction to  the Heisenberg fan.

The operators $M_\pm$ have the following relevant property. If $f$  is radial  and $(1+\mathcal A )f$ is
integrable   on $\Hn$   then~(see \cite{BJR})
 \begin{equation}\label{Mpm}
 \gel({{{\mathcal A}f)}}=M_+ ( \gel{ f })
 \qquad{\text{ and}}\qquad  \gel({{\bar{\mathcal A}f}})=-M_- ( \gel{ f})
 .
\end{equation}
 One can verify that 
$$
\dualR{M_+(\gel f)\mu}{\gel h}=-\dualR{(\gel f)\mu}{M_-(\gel h)}
\qquad \forall f, h\in \schwkrad.
$$
Hence by Theorem~\ref{nostro}
 
 \begin{equation}\label{Mtrasposto}
\int_\cfan (M_+\vp)\,\psi\, d\mu=
-\int_\cfan \vp\, (M_-\psi)\, d\mu
\qquad \forall \vp,\psi\in \cS(\R^2).
\end{equation}
 
According to~\pref{Mtrasposto},
when $U$ is in $\cS_0'(\cfan)$ we define the distribution $M_+U$ by
$$
\dualR{M_+U}{\psi}
=-\dualR{U}{M_-\psi}
\qquad\forall \psi\in \cS(\R^2)
$$
and similarly 
we define $M_-U$
with $M_+$ and $M_-$ interchanged.

Clearly if $\psi_{|_\cfan}=0$ then $(M_-\psi)_{|_\cfan}=0$, so that $M_+U$ is in $\cS_0'(\cfan)$.

Moreover it is easy to verify that \pref{Mpm} extends to distributions, i.e.
\begin{equation}\label{emmepiu}
 \gel({{{\mathcal A}\Lambda)}}=M_+ ( \gel{ \Lambda })
 \qquad{\text{ and}}\qquad  \gel({{\bar{\mathcal A} \Lambda}})=- M_- ( \gel{ \Lambda})
 \qquad \forall \Lambda \in \cS'(\Hn)
 .
\end{equation}
  Finally, given a distribution in $\cS_0'(\cfan)$
of the form $F\mu$, where $F$ is smooth and slowly growing
on $\R^2$,
we note that for all $\psi$ in $\cS(\R^2)$,
\begin{align*}
\dualR{M_+(F\mu)}{\psi}
&=
-\dualR{F\mu}{M_- \psi}
\\
&=-\int_\cfan F\, M_- \psi\, d\mu
\\
&=\int_\cfan M_+ F\, \psi\, d\mu
=\dualR{(M_+F)\,\mu}{\psi},
\end{align*}
therefore
\begin{equation}
\label{emmepiuF}
M_+(F\mu)=(M_+ F)\,\mu\ .
\end{equation}
 
For later use, we prove the following estimate.

\begin{lemma}\label{potenzeM+}
Let  $a$ be a positive integer, then
for every $\psi$ in 
 $\cD(\R^2)$
  with support in the set 
$ \set{(\xi,\la)\in \R^2\,\,:\,\, |\xi|\leq \suppfan}$
$$
\|M_\pm^a \psi \|_{L^1(\cfan)}\leq C_a\, (1+\suppfan)^{a+\dimH+2}
\sum_{s,r=0}^{2a}\|\partial_\la^s\partial_\xi^r \psi \|_{L^\infty(\R^2)}. 
$$
\end{lemma}
\begin{proof} 
It is enough to prove the statement for $M_+$, since   
$M_-\check\psi = -\left[M_+\psi \right]\!\!\check{\phantom f}$, where
$\check\psi(\xi,\la)=\psi(\xi,-\la)$. 
 Let 
 $W$ denote the operator acting on  a smooth function $\psi $ on $\R^2$ by
 $$
 W \psi(\xi,\la) = 2\int_0^1\partial^2_\xi \psi(\xi+2\la t,\la )(1-t)\,dt.
 $$
For every $ j\geq 0$ let $\eta_j$ be the function 
and let  $V_j$  be the operator defined by
$$
\eta_j(\xi,\la)=\xi+(2j+\dimH)\la
\qquad V_j= \partial_\la-(2j+\dimH) \partial_\xi.
$$  
With this notation $M_+=V_0-\eta_0W$. 
Moreover, as proved in~\cite[{Lemma 4.5}]{ADR2}, for every positive integer $a$,
\begin{equation}
M_+^a=V_0^a+\sum_{k=1}^a \eta_0\cdots \eta_{k-1}\, D_{k,a},
\end{equation}
where  $D_{k,a}$ is a polynomial  in $V_0,\ldots,V_{k},W$ of degree $a$ 
such that in each monomial the operator $W$ appears  $k$ times.

Let $\psi $ be in   $\cD(\R^2)$
 with support in the set 
$ \set{(\xi,\la)\in \R^2\,\,:\,\, |\xi|\leq \suppfan}$.
Then it is easy to see that 
 $\supp D_{k,a} \psi\subseteq  \set{(\xi,\la)\in \R^2\,\,:\,\, |\xi|\leq c\,\suppfan}$, 
 with $c$ depending on $a$. Therefore, using~\pref{integrabilita},
  \begin{align*}
\|  M_+^a \psi\|_{L^1(\cfan)}&
 \leq
 \|V_0^a \psi\|_{L^1(\cfan)}+\sum_{k=1}^a \| \eta_0\cdots \eta_{k-1}\, D_{k,a} \psi\|_{L^1(\cfan)}
  \\
 &\leq
 C_a\, (1+\suppfan)^{a+n+2}\left(
\sum_{r+s\leq a} \|\partial_\la^s\partial_\xi^r \psi \|_{L^\infty(\R^2)}
+
\,\sum_{k=1}^a \| D_{k,a} \psi\|_{L^\infty(\R^2)}
\right).
\end{align*}

We complete the proof by showing that 
$$
\|D_{k,a} \psi \|_{L^\infty(\R^2)}\leq C_a\, 
\sum_{s+r\leq 2a}\|\partial_\la^s\partial_\xi^r \psi \|_{L^\infty(\R^2)}
\qquad k=1,2,\ldots,a,
$$
by induction on $a$. Indeed, the case $a=1$ is trivial since
$$
\|W \psi \|_{L^\infty(\R^2)}\leq 
2\int_0^1\|\partial^2_\xi \psi \|_{L^\infty(\R^2)}(1-t)\,dt
\leq C\, 
\|\partial_\xi^2 \psi \|_{L^\infty(\R^2)} .
$$
If $a>1$ then either $D_{k,a}=D_{k-1,a-1}\,W$ or $D_{k,a}=D_{k,a-1}\,V_j$, 
for some $j$ and $k\leq a-1$. The second case is trivial.
If $D_{k,a}=D_{k-1,a-1}\,W$, we note that  by induction on $s$ it is easy to verify that 
$$
\partial_\la^s\partial_\xi^r W \psi =
 2\sum_{k=0}^{s}\binom{s}{k}
 \int_0^1 \partial_\la^{s-k}\partial_\xi^{r+2+k}\, \psi(\xi+2\la t,\la )\,
 (2t)^{k}\,(1-t)\,dt.
$$
Therefore
  \begin{align*}
\|D_{k-1,a-1}\,W \psi \|_{L^\infty(\R^2)}&\leq C_a\, 
\sum_{s+r\leq 2a-2}\|\partial_\la^s\partial_\xi^r W \psi \|_{L^\infty(\R^2)}
\\ &\leq C_a\, 
\sum_{s+r\leq 2a-2}\sum_{k=0}^{s}\|\partial_\la^{s-k}\partial_\xi^{r+2+k} \psi \|_{L^\infty(\R^2)}
\\ &\leq C_a\, 
\sum_{s+r\leq 2a}\|\partial_\la^s\partial_\xi^r   \psi \|_{L^\infty(\R^2)}
.\qedhere
\end{align*}
\end{proof}

\bigskip

\section{Real Paley--Wiener results for the spherical transform}

Suppose that  $\Lambda$ is a radial tempered distribution on $\Hn$.
Motivated by~\cite{Nils}, we define 
$R(\Lambda)$  in $[0,\infty]$ by
$$
R(\Lambda)=
{\rm max}\set{|x|\, :\, x\in {\rm supp}\, \Lambda}
$$
and we call $R(\Lambda)$ the radius of the support of the distribution~$\Lambda$.

The purpose of this section is to prove real Paley--Wiener Theorems for the spherical transform;
we start with a characterization of compactly supported radial
distributions and then we specialize these results to square integrable
radial functions and Schwartz radial functions. 
When a distribution $U$ on $\R^2$ is of the form $U=F_U\,\mu$ with $F_U$
a (smooth)  
 function on $\R^2$, by abuse of notation 
we shall also denote by $U$ the associated function $F_U$.

Our first characterization reads as follows.

\begin{theorem}\label{unoequiv}
Let $\Lambda$ be in $\schwkradp$. The following conditions are equivalent.
\begin{enumerate} 
\item
$R(\Lambda)$ is finite;

\item  $\gel \Lambda $ is the restriction to $\cfan$ of
  a smooth function on $\R^2$ and for every $p$ in $[1,\infty]$ there exists $\beta>0$ such that
$$
\limsup_{j\to\infty}
\|(1+\xi)^{-\beta}\, M_+^j \gel \Lambda\|_{L^p(\cfan)}^{1/j}<\infty;
$$

\item 
 for every large $j$  the distribution
$  M_+^j \gel \Lambda$ 
%coincides with a smooth slowly growing function  on $\R^2$
is the restriction to $\cfan$ of
  a smooth function on $\R^2$
     and
       there exist $\beta>0$ and $p$ in $[1,\infty]$
such that
$$
\liminf_{j\to\infty} 
\|(1+\xi)^{-\beta}\, M_+^j \gel \Lambda\|_{L^p(\cfan)}^{1/j}
<\infty .
$$

\end{enumerate}
Moreover, if any of these conditions is satisfied, then   for every $p$ in $[1,\infty]$ there exists $\beta>0$ such that
\begin{equation}\label{Rspettr}
\lim_{j\to\infty} 
\|(1+\xi)^{-\beta}\, M_+^j\gel \Lambda\|_{L^p(\cfan)}^{1/j}= R(\Lambda)^2.
\end{equation}
\end{theorem}

Since $M_-\psi = -\left[M_+\check\psi \right]\check{\phantom f}$, where
$\check\psi(\xi,\la)=\psi(\xi,-\la)$,
 we have also a corresponding analogue
 with $M_-$ in place of $M_+$.  

 The proof of Theorem~\ref{unoequiv} is given after some preliminary results,
 the first of which is the following technical lemma.

 \begin{lemma}\label{lemmafj}
Let $R>0$ and $j$ be a positive integer.
 Suppose that $f$ is a smooth function on $\Hn$
with compact support in the set $\{x\in \Hn\,:\, |x|>R\}$
and let $f_j=\bar\cA^{-j}\, f$.\\
Then for every $N$ in $\N$  
\begin{align*}
\|(1+\xi)^N\, \gel f_j\|_{L^\infty(\cfan)}
&\leq
 C_{N}\,j^{2N}\,R^{-2j}\, 
\max_{h\leq N}
\sum_{{k+\deg J=2h}}
\|{\bar\cA}^{-k}\,{\op D}^{J} f\|_{L^1(\Hn)}
\end{align*}
\end{lemma}
 \begin{proof} Note that since $f$ is supported away from the origin,
the function $f_j= \bar\cA^{-j}\,f$ is again smooth and compactly supported.
Moreover,
\begin{align*}
\|(1+\xi)^{N}\,\gel{f_j}\|_{L^{\infty}(\cfan)}
&= 
\|\gel\bigl({(I+L)^{N}f_j}\bigr)\|_{{L^{\infty}(\cfan)}}
\\
&\leq 
\|(I+L)^Nf_j \|_{L^1(\Hn)}
.
\end{align*}
Clearly $(I+L)^Nf_j=\sum\binom{M}{h}L^hf_j$ and
by the Leibniz rule
  
\begin{align*}
\|(I+L)^Nf_j \|_{L^1(\Hn)}
&
\leq C_N \,\max_{h\leq N}\|L^hf_j\|_{L^1(\Hn)}
\\
&\leq C_N\,
\max_{h\leq N}
\sum_{{\deg I+\deg J=2h}}
\|({\op D}^I {\bar\cA}^{-j})\,({\op D}^{J} f) \|_{L^1(\Hn)}
\\
&\leq C_N\,
\max_{h\leq N}
\sum_{{\deg I+\deg J=2h}} j^{|I|}
\|({\bar\cA}^{-j-\deg I})\,({\op D}^{J} f) \|_{L^1(\Hn)}
\\
&\leq C_N\,j^{2N}\,R^{-2j}\, 
\max_{h\leq N}
\sum_{{\deg I+\deg J=2h}}
\|{\bar\cA}^{-\deg I}\,{\op D}^{J} f\|_{L^1(\Hn)}.
\qedhere
\end{align*}
\end{proof}

Now we note that the spherical transform of radial compactly supported distributions
satisfies a pointwise estimate on the Heisenberg fan $\cfan$.

\begin{proposition}
\label{puntuale}
Let $\Lambda$ be a radial compactly supported distribution of order $N$ on $\Hn$. Then for every    $R>R(\Lambda)$
there exists a constant $C=C_R>0$ 
 such that for every $j$ in $\N$
  \begin{equation}\label{claim}
|M_+^j \widehat \Lambda(\xi,\la)|\leq C\,  
 R^{2j}\, (1+\xi)^{N/2}
\qquad
\forall (\xi,\la)\in\cfan .
\end{equation}
\end{proposition}

\begin{proof} We have already proved in Proposition~\ref{distolo}
that $\gel\Lambda=\widehat\Lambda\, \mu$
and that $\widehat\Lambda$ extends to an entire function, so 
$\widehat\Lambda$ is in  $C^\infty(\cfan)$.
Moreover by equations~\pref{emmepiu}
and~\eqref{emmepiuF}
\begin{align*}
(M_+^j\widehat\Lambda)\, \mu
&=M_+^j(\widehat\Lambda\, \mu) 
=M_+^j\gel\Lambda 
=\gel({\cA^j \Lambda}) 
=\widehat{\cA^j \Lambda}\,\mu,
\end{align*}
therefore $M_+^j\widehat\Lambda=\widehat{\cA^j \Lambda}$.

Let
 $R>R(\Lambda)$ and choose $R_1$ such that $R>R_1>R(\Lambda)$. 
Suppose 
that $g$ is a radial test function on $\Hn$ such that $g(x)=1$
when $x$ is in the support of $\Lambda$ and $g(x)=0$ if $|x|> R_1$.
Then for all $(\xi,\la)$ in $\fan$,
\begin{align*}
|M_+^j \widehat \Lambda(\xi,\la)|
&=| \widehat {\cA^j \Lambda}(\xi,\la)|
=| \widehat {g\,\cA^j\Lambda}(\xi,\la)|
\\
&=|\dualH{g\,\cA^j\Lambda}{\check{\Phi}_{\xi,\la}} |
\\
&=| \dualH{\Lambda}{g\,\cA^j\Phi_{\xi,-\la}} |
\\
&\leq C\, \sum_{\deg I\leq N} 
\| {\op D}^I (g\,\cA^j\Phi_{\xi,-\la})\|_{L^\infty(\Hn)}.
\end{align*}
We conclude 
by the Leibniz rule
and Lemma~\ref{stimasferiche} that
$$
|M_+^j \widehat \Lambda(\xi,\la)|\leq C\, (1+j)^N\, R_1^{2j}\, (1+\xi)^{N/2}
\qquad
\forall (\xi,\la)\in\fan,
$$
which, since $R>R_1$ and by the smoothness of $\widehat\Lambda$, implies~\pref{claim}.
\end{proof}

\bigskip
Conversely, it is easy to deduce that a radial tempered distribution is compactly
supported
when a certain limit is finite.
\begin{proposition}
\label{due}
Let $\Lambda$ be in $\schwkradp$.   
Suppose that there exists  $J$ in $\N$ such that
  for every 
 $j\geq J$ the distribution $ M_+^j\gel \Lambda$ is of the form $G_j\,\mu $, where $G_j$ is 
 a locally integrable function with respect to $\mu$. Then for every $N$ in $\N$
and every $p$ in $[1,\infty]$
$$
\liminf_{j\to\infty} 
\|\,(1+\xi)^{-N
}\, G_j\,\,\|_{L^p(\cfan)}^{1/j}
\geq R(\Lambda)^2.
$$ 
\end{proposition}

\begin{proof}
Suppose that $R(\Lambda)>0$ and let $0<\eps<R(\Lambda)/2$.
Then  
we may find a smooth function
$f$ with compact support in the set
$$
\{x\in \Hn\, :\, R(\Lambda)-\eps<|x|< R(\Lambda)+\eps \}
$$
 such that
$\dualH{\Lambda}{\check f}\neq 0$.
As in the previous lemma, 
the function~$f$ is supported away from the origin and we let
$f_j= \bar\cA^{-j}\,f$.
By\pref{tildeT}  and\pref{emmepiu}
\begin{align*}
|\dualH{\Lambda}{\check f}|
&=|\dualH{\Lambda}{ \cA^{j}\, \cA^{-j}\,\check f}|
=|\dualH{\Lambda}{ \cA^{j}\,\check f_j}|
\\
&=|\dualH{\cA^{j}\,\Lambda}{\check{f_j}}|
=|\dualR{\gel{(\cA^{j}\,\Lambda)}}{\gel{ {f_j}}}|
\\
&=|\dualR{M_+^{j} \gel{\Lambda}}{\gel{ {f_j}}}|
=|\dualR{G_j\, \mu}{\gel{ {f_j}}}|
\\
&\leq 
\| (1+\xi)^{-N}\,G_j\|_{L^p(\cfan)}\,
\|(1+\xi)^{N}\,\gel{ f_j}\|_{L^{p'}(\cfan)}.
\end{align*}
 In the case where $\| (1+\xi)^{-N}\,G_j\|_{L^p(\cfan)}=\infty$ for all $j$,
 there is nothing to prove. Otherwise, since $|\dualH{\Lambda}{\check{f}}|\neq0$,
$$
\liminf_{j\to\infty}
\| (1+\xi)^{-N}\,G_j \|_{L^p(\cfan)}^{1/j}
\geq \liminf_{j\to\infty}
\left(\frac{|\dualH{\Lambda}{\check{f}}|}{\|(1+\xi)^{N}\,\gel{{f_j}}\|_{L^{p'}(\cfan)}}
\right)^{1/j}
\!= \liminf_{j\to\infty}\|(1+\xi)^{N}\,\gel{f_j}\|_{L^{p'}(\cfan)}^{-1/j}.
$$
Since there exists $M$ in $\N$ such that
$\xi\mapsto (1+\xi)^{N-M}$ is in $L^{p'}(\cfan)$, by Lemma~\ref{lemmafj}
we conclude that
\begin{align*}
\|(1+\xi)^{N}\,\gel{f_j}\|_{L^{p'}(\cfan)}
&\leq 
\|(1+\xi)^{N-M}\|_{L^{p'}(\cfan)}\,
\|
(1+\xi)^{M}\,
\gel{f_j}
\|_{{L^{\infty}(\cfan)}}
\\
&
\leq C\,\,j^{2M}\,(R(\Lambda)-\eps)^{-2j}\ ,
\end{align*}
and the thesis follows easily.

When $R(\Lambda)=\infty$ we use the same arguments to show that \hbox{$\liminf\limits_{j\to\infty}
\| (1+\xi)^{-N}\,G_j \|_{L^p(\cfan)}^{1/j}\geq R$} for every $R>0$.
\end{proof}

Putting together Proposition~\ref{puntuale} and Proposition~\ref{due},
we obtain the following criterion, by which we can measure
the size of the support of a radial compactly supported distribution.

\begin{corollary}
\label{prop68}
Let $\Lambda$ be a radial compactly supported distribution of order $N$. Then 
$$
\lim_{j\to\infty} 
\|(1+\xi)^{-N/2}\, M_+^j\widehat\Lambda\|_{L^\infty(\cfan)}^{1/j}= R(\Lambda)^2.
$$
\end{corollary}

\begin{proof}
From the pointwise estimate~\pref{claim}, we deduce that for every $R>R(\Lambda)$
$$
\limsup_{j\to\infty}
\|(1+\xi)^{-N/2}\, M_+^j\widehat\Lambda\|_{L^\infty(\cfan)}^{1/j}\leq R^2,
$$
therefore $\displaystyle{\limsup_{j\to\infty}
\|(1+\xi)^{-N/2}\, M_+^j\widehat\Lambda\|_{L^\infty(\cfan)}^{1/j}\leq R(\Lambda)^2}$.
The thesis follows by Proposition~\ref{due}.
\end{proof}

 \begin{proof}[Proof of Theorem~\ref{unoequiv}]
If $\Lambda$ is compactly supported  and 
of order $N$ then  by Proposition~\ref{distolo} it coincides with 
  a smooth slowly growing function $G$ on $\R^2$. If 
$\beta>0$ is such that
$(1+\xi)^{N/2-\beta}$ is in ${L^p(\cfan)}$, we have
\begin{equation}
\label{casop}
\|(1+\xi)^{-\beta}\, M_+^jG\|_{L^p(\cfan)} 
\leq 
\|(1+\xi)^{N/2-\beta} \|_{L^p(\cfan)}
\|(1+\xi)^{-N/2}\, M_+^jG\|_{L^\infty (\cfan)}.
\end{equation}
Hence by  Corollary~\ref{prop68} we have that    (1) $\Rightarrow$ (2).
The implication (2) $\Rightarrow$ (3) is trivial and the implication   
   (3) $\Rightarrow$ (1) is a consequence of Proposition~\ref{due}. 
   Finally \eqref{Rspettr} follows by \eqref{casop},  Proposition~\ref{due} and Corollary~\ref{prop68}.
 \end{proof}

   \subsection{Square-integrable functions}

\begin{theorem}\label{teoL2}\label{tre}
Suppose that for every $j\geq 0$ the function
$M_+^j\psi$ is in $L^2(\cfan)$. Then the function~$f$ such that $\gel f=\psi$
is in $L^2(\Hn)$and
$$
\lim_{j\to\infty} \|\,M_+^j\psi\,\,\|_{L^{2}(\cfan)}^{1/j}= R(f)^2.
$$ 
\end{theorem}

\begin{proof} By Proposition~\ref{due}
it is enough to check that 
$\limsup_{j\to\infty} \|M_+^j\psi\|_{L^{2}(\cfan)}^{1/ j}\leq R(f)^2$,
and this is easily established by using the Plancherel formula. Indeed,
when $R(f)$ is finite,
\begin{align*}
\|M_+^j\psi \,\|_{L^{2}(\cfan)}
&= \|\cA^j\,f\|_{L^{2}(\Hn)}
\leq   R(f)^{2j} \, \|f\|_{L^{2}(\Hn)}.
\qedhere
\end{align*}
\end{proof}

\begin{corollary}\label{L2}
Let $R\geq 0$. Then  
$\gel$
is a bijection from the space $L^2_{\text{\rm rad},R}(\Hn)$ 
of square integrable radial functions~$f$ such that $R(f)\leq R$
onto $\{\psi\in L^2(\cfan)
\,\,: 
\lim_{j\to\infty} 
\|\,\,M_+^j \psi\,\,\|_{L^2(\cfan)}^{1/j}
\leq R^2
\}$
.
\end{corollary}

\subsection{Schwartz functions}
 
The purpose of this subsection is to prove the following characterization.

\begin{theorem}\label{sch}
Let $f$ be in $\schwkrad$. The following conditions are equivalent.
\begin{enumerate}

\item  $R(f)$ is finite;

\item  for every $h\geq 0$ and every $p$ in $[1,\infty]$,
 $\limsup_{j\to\infty} 
\|\,\xi^h\,M_+^j\gel f\,\,\|_{L^p(\cfan)}^{1/j}
$ is finite;

\item  there exists  $p$ in $[1,\infty]$ such that
 $\liminf_{j\to\infty} 
\|\,M_+^j\gel f\,\,\|_{L^p(\cfan)}^{1/j}
$ is finite.

\end{enumerate}
Moreover, if any of these conditions is satisfied, then for every $h\geq 0$ and every $p$ in $[1,\infty]$,
$$
\lim_{j\to\infty} 
\|\,(1+\xi)^h\,M_+^j\gel f\,\,\|_{L^p(\cfan)}^{1/j}
=R(f)^2.
$$
\end{theorem} 

Note that the implication $(2)\Rightarrow (3)$ is trivial,
and that $(3)\Rightarrow (1)$ follows from Proposition~\ref{due}.
In the next proposition we prove the implication $(1)\Rightarrow (2)$.
\begin{proposition} 
\label{uno}
Suppose that $f$ is a radial Schwartz function on $\Hn$.
 Then
for every $h\geq 0$ and every $p$ in $[1,\infty]$
$$
\limsup_{j\to\infty} 
\|\,(1+\xi)^h\,M_+^j\gel f\,\,\|_{L^p(\cfan)}^{1/j}
\leq R(f)^2.
$$ 
\end{proposition}

\begin{proof}  
If $R(f)=\infty$ there is nothing to prove. If $R(f)=0$, then $f=0$
and the conclusion is again trivial.
We therefore suppose that $R(f)$ is positive. Note that
$$
\xi^h\,M_+^j\gel f(\la,\xi)
=
\gel\bigl({L^h\mathcal{A}^{j} f}\bigr) 
$$
and when $j\geq 2h$,  by the Leibniz rule
\begin{align*}
|L^h \mathcal{A}^{j}\, f|
&=\left| 
\sum_{{\deg I+\deg J=2h}}
c_{h,I,J}\,({\op D}^I \mathcal{A}^{j})\,({\op D}^{J} f)
\right|
\\ 
&\leq 
\sum_{{\deg I+\deg J=2h}}|c_{h,I,J}|
\, j^{|I|} \,|\mathcal{A}^{j-\deg I}|\,|{\op D}^{J} f|.
\end{align*}
 Therefore 
\begin{align*}
\|\, \xi^h\,M_+^j\gel f\,\,\|_{L^\infty(\cfan)}
&\leq \|L^h\mathcal{A}^{j}\, f\|_{L^1(\Hn)}
\\
&
\leq C_h\, j^{2h}\,\sum_{q\leq 2h}
\max_{\deg J=2h-q} \|\mathcal{A}^{j-q}\,{\op D}^{J} f\|_{L^1(\Hn)}
\\
&
\leq C_h\, j^{2h}\,\sum_{q\leq 2h}\, R(f)^{2j-2q}\, 
\max_{\deg J=2h-q}\|{\op D}^{J} f\|_{L^1(\Hn)}
\\
&
= C_{f,h}\,j^{2h}\, R(f)^{2j}.
\end{align*}
 
We note that
for a sufficiently big integer $M$ the function $(\la,\xi)\mapsto
(1+\xi)^{-M}$ is in $L^{p}(\cfan)$, so that
\begin{align*}
\|\,(1+\xi)^h\,M_+^j\gel f\,\|_{L^{p}(\cfan)}
&\leq C\,\,
\|(1+\xi)^{M+h}\,M_+^j\gel f\,\|_{L^{\infty}(\cfan)}
\\
&\leq C_{f,M,h}\, \left( 1 +j^{2M+2h}\right)\,R(f)^{2j},
\end{align*}
and  
taking the $j$-th root, the desired inequality follows.
\end{proof}

\begin{corollary}
Suppose that $f$ is a radial Schwartz function on $\Hn$ and let $1\leq p\leq \infty$.  Then
for every $h$ in $\N$
$$
\lim_{j\to\infty} \|\,(1+\xi)^h\,M_+^j\gel f\,\,\|_{L^{p}(\cfan)}^{1/j}= R(f)^2.
$$ 
\end{corollary}

\begin{proof} Since $\|(1+\xi)^h\,M_+^j\gel f\|_{L^{p}(\cfan)}
\geq \|\,M_+^j\gel f\,\|_{L^{p}(\cfan)}$, by Proposition~\ref{due} we obtain
$$
\liminf_{j\to\infty} \|\,(1+\xi)^h\,M_+^j\gel f\,\,\|_{L^{p}(\cfan)}^{1/j}
\geq \liminf_{j\to\infty}\|\,M_+^j\gel f\,\,\|_{L^{p}(\cfan)}^{1/j}\geq R(f)^2.
$$
The thesis follows from Proposition~\ref{uno}.
\end{proof}

\section{Paley--Wiener theorems for the inverse spherical transform }

In this section we describe the inverse spherical transform
of compactly supported distributions in $\cS_0'(\cfan)$.

Given a compactly supported distribution in $\cS'(\R^2)$,
we define the function $f_U$ on the Heisenberg group by
$$
f_U(z,t)=\dualR{U}{\Phi_{(\cdot)}(z,t)}\qquad \forall  (z,t)\in\Hn.
$$ 
An easy consequence of 
 Lemma~\ref{olosferiche} is the following.

\begin{lemma}\label{oloinv}
Let $U$ be a compactly supported distribution in $\cS'(\R^2)$.
 Then the function
$$
(x,y,t)\mapsto f_U(x+iy,t)=\dualR{U}{\Phi_{(\cdot)}(x+iy,t)}
$$ 
extends to a holomorphic function   on $\C^{2\dimH+1}$.
\end{lemma}

If $U$ is in $\cS'(\R^2)$, define
 $$
\suppfan (U)=
{\rm max}\set{|\xi|\, :\, (\xi,\la)\in {\rm supp}\, U},
$$
 so that
  a distribution $U$ in $\cS_0'(\cfan)$  is compactly supported 
 if and only if $\rho(U)$ is finite.

In the next proposition we prove that if $U$ is 
a compactly supported distribution in $\cS_0'(\cfan)$ then the function $f_U$
is a slowly growing function on $\Hn$ and it coincides with 
the inverse spherical transform of $U$.

\begin{proposition}\label{invtemp}
Let $U$ be a compactly supported distribution in $\cS_0'(\cfan)$ and let, as before,
$$
f_U(z,t)=\dualR{U}{\Phi_{(\cdot)}(z,t)}\qquad \forall  (z,t)\in\Hn.
$$ 
Then $  U=\gel(f_U\,m)$ and $f_U$ is a  slowly growing function on $\Hn$ together  with all its derivatives. Moreover, for every $\rho>\suppfan(U)$
there exist $C=C_\rho$ and $M$ such that for all $j\geq 0$ 
\begin{equation}
\label{e:Lj}
 \left|
     L^j f_U(z,t)
     \right|
     \leq C \,(1+j)^k\, \rho^j\, \big(1+|(z,t)|\big)^M
     \qquad \forall (z,t)\in \Hn,
\end{equation}
where $k$ is the order of $U$.
\end{proposition}

\begin{remark}
Observe that if $U$ is 
a  distribution in $\cS'(\R^2)$ with compact support in $\cfan$,
then the function $f_U$
may not be slowly growing.
Indeed let $U=\partial_\xi\delta_{(\dimH,1)}$ where $\delta_{(\dimH,1)}$ is the Dirac measure
at the point $(\dimH,1)$ in $\cfan$. Then, reasoning as in Lemma~\ref{olosferiche},
when $(z,t)$ is in $\Hn$
$$
f_U(z,t)=-\partial_\xi\Phi_{\xi,\la}{\,}_{|_{(\dimH,1)}}(z,t)
=\frac{e^{it}\, e^{-|z|^2/4}}{2}\sum_{k=1}^\infty \frac{(|z|^2/2)^k}{k\, (\dimH)_k}.
$$
Since $k<k+ \dimH$ and $(\dimH)_k\leq (\dimH+k-1)!$ we obtain
when $|z|$ is large
$$
|f_U(z,t)|>\frac{e^{-|z|^2/4}}{2}\sum_{k=1}^\infty \frac{(|z|^2/2)^k}{(\dimH+k)!}
\sim \frac{e^{|z|^2/4}}{2(|z|^2/2)^\dimH}.
$$
This is due, much as in Subsection~\ref{sec:olo}, to the fact that
the holomorphic extension of spherical functions does not satisfy good estimates
away from the Heisenberg fan.
The main point in the proof of Proposition~\ref{invtemp} is that,
according to formula~\pref{numero},
if $U$ is in  $\cS_0'(\cfan)$ one is allowed to choose a different extension.
 \end{remark}

\begin{proof}  By Theorem~\ref{nostrodist} there exists $\Lambda$ in $\schwkradp$ such that $\gel  \Lambda = U $. 
Let $g$ be in $\cD(\Hn)$,  then 
$$
\begin{aligned}
\langle f_U,g\rangle_{H_n}&=
\int_{H_n}f_U(z,t)\,g(z,t)\,dz\,dt 
\\
&=\int_{H_n}\dualR{U}{\Phi_{(\cdot)}(z,t)}\,\,g(z,t)\,dz\,dt\\
&=\Big\langle U,\int_{H_n}\Phi_{(\cdot)}(z,t)\,g(z,t)\,dz\,dt\Big\rangle_{\R^2}\\
&=\dualR{U}{\cG \check g}\\
&=\langle \Lambda,g\rangle_{H_n}\ .
\end{aligned}
$$
Hence the distribution  $\Lambda$ coincides with the function $f_U$, which is smooth.

We first prove the estimate~\pref{e:Lj}.
Fix $(z,t)$ in $\Hn$. 
Let $k$ be the order of $U$, let  $\rho>\rho(U)$ and   
denote by  $B_\rho$ the ball of radius $\rho$ in $\R^2$. Then for every $j\geq 0$
 \begin{align}
   \nonumber  \left|
     L^j f_U(z,t)
     \right|&=
     \left|\dualR{ 
    U }{ {L}^j \Phi_{(\cdot)}(z,t)}
     \right|
     =
      \left|\dualR{\,U}{ \xi^j \Phi_{(\cdot)}(z,t)}
     \right|
     \\ \label{numero} &
           \leq
      C \inf\set{ 
     \|
     \xi^j \varphi^{z,t}
      \|_{C^k(B_\rho)}\,:\,\varphi^{z,t}\in C^k(\R^2),\quad
 {\varphi^{z,t}}_{|_{\cfan\cap B_\rho}}=\Phi_{(\cdot)}(z,t)
}
     \\ &
     \nonumber\leq
     C_{\rho}\, (1+j)^k\, \rho^j\inf\set{ 
      \|\varphi^{z,t}
      \|_{C^k(B_\rho)}\,:\,\varphi^{z,t}\in C^k(\R^2),\quad
 {\varphi^{z,t}}_{|_{\cfan\cap B_\rho}}=\Phi_{(\cdot)}(z,t)
}
\end{align}

In order to obtain the desired estimate we shall choose a suitable extension  $\varphi^{z,t}$ of $\Phi_{(\cdot)}(z,t)$.

Let $\psi$ be a smooth function on $\R^2$ with compact support such that $\psi_{|{B_\rho} }=1$. By Theorem~\ref{nostro} there exists $u$ in $\schwkrad$ such that $\cG u=\psi_{|\cfan}$. If  
$\nu_{(z,t)}$ denotes the measure defined by 
$$
\int_{\Hn} g(w,s)\, d\nu_{(z,t)}(w,s)=
\int_{U(n)} g(kz,t)\, dk \qquad\forall g\in C_c(\Hn),
$$
then for every $(\xi,\la)$ in $B_\rho\cap \cfan$
$$
\Phi_{\xi,\la}(z,t)
=\gel{\check\nu_{(z,t)}}(\xi,\la)=
\gel{\check\nu_{(z,t)}}(\xi,\la)\, \psi(\xi,\la)=
\gel\left(\check\nu_{(z,t)}\ast u
\right)(\xi,\la)
.
$$
Since $\check\nu_{(z,t)}\ast u$ belongs to $ \schwkrad$, then by Theorem~\ref{nostro} there exist
$\varphi^{z,t}$ in $ \cS(\R^2)$ and  $M\geq 0$ such that 
$$
\varphi^{z,t}(\xi,\la)=\gel\left(\check\nu_{(z,t)}\ast u
\right)(\xi,\la)
\qquad \forall(\xi,\la)\in   \cfan
$$
and 
$$
 \|\varphi^{z,t}
      \|_{C^k(B_\rho)}\leq C \, \|\check\nu_{(z,t)}\ast u\|_{(M)}.
$$
Moreover
$$
\varphi^{z,t}(\xi,\la)= \Phi_{(\xi,\la)}(z,t)
\qquad \forall(\xi,\la)\in  B_\rho\cap\cfan.
$$
If 
$ \tau_{(w,s)}u(w',s')=u\big((w,s)^{-1}(w',s')\big)$ denotes the left translation, then
 \begin{align*}
\|\varphi^{z,t}
      \|_{C^k(B_\rho)}
      &\leq {  C \, \|\nu_{(z,t)}\ast u\|_{(M)}
}
\\
&=C \, 
{ \left\|\int_{\Hn} \tau_{(w,s)}u\, d\nu_{(z,t)}(w,s)\right\|_{(M)}}
 \\
&\leq C \, 
  \int_{\Hn} \|\tau_{(w,s)}u\|_{(M)}\, d\nu_{(z,t)}(w,s)
   \\
&\leq C \, 
{   \int_{\Hn} (1+|w|^4+s^2)^{M/4}\, d\nu_{(z,t)}(w,s)
}
  \\
  &=    C \, \big(1+|(z,t)|\big)^{M}.
\end{align*}
Therefore there exists $M$ such that for all $j\geq 0$ 
$$
 \left|
     L^j f_U(z,t)
     \right|
     \leq C \,(1+j)^k\, \rho^j\, \big(1+|(z,t)|\big)^{M}
\qquad\forall (z,t)\in \Hn.
$$

The proof above can be adapted to prove that for every differential operator ${\op D}^I  $ 
of the form~\pref{monomi}
there exists $M>0$ such that   
$$
|{\op D}^I f_U(z,t)|\leq C\, (1+|(z,t)|)^M
\qquad \forall (z,t)\in \Hn.
$$
Indeed, note that
$$
 {\op D}^I f_U(z,t)=\dualR{U}{{\op D}^I\Phi_{(\cdot)}(z,t)}\ ,
$$
therefore
\begin{align*}
| {\op D}^I f_U(z,t)|
&\leq C \inf\set{\|\varphi^{z,t,I}\|_{C^k(B_\rho)}\,:\,\varphi^{z,t,I}\in C^k(\R^2)\quad
{\varphi^{z,t,I}}_{|_{\cfan\cap B_\rho}}= {\op D}^I \Phi_{(\cdot)}(z,t)}.
\end{align*} 
Fix $(z,t)$ in $\Hn$ and consider the distribution
 ${\op D}^I_{(z,t)}\nu_{(z,t)}$ defined by the rule
$$
\dualH{{\op D}^I_{(z,t)}\nu_{(z,t)}}{\varphi}
={\op D}^I\left(\int_K \varphi(kz,t)\, dk \right)
$$
Then ${\op D}^I_{(z,t)}\nu_{(z,t)}$ is a radial distribution
 supported in the orbit of $(z,t)$,
hence it has compact support. So, for $\psi$ and $u$ as above,
${\op D}^I_{(z,t)}\check\nu_{(z,t)}*u$ is in $\schwkrad$ and
by~\cite[Proposition~3.2]{ADR} there exists
$\varphi^{z,t,I}$ in $C^k(\R^2)$ and $M$ such that
$$
\varphi^{z,t,I}_{|_\cfan}=\cG({\op D}^I_{(z,t)}\check\nu_{(z,t)}*u)
\qquad\qquad
\|\varphi^{z,t,I}\|_{C^k(B_\rho)}\leq C\,\|{\op D}^I_{(z,t)}\check\nu_{(z,t)}*u\|_{(M)}
$$
Since $\gel u_{|_{\cfan\cap B_\rho}}=\psi_{|_{\cfan\cap B_\rho}}=1$,
$$
\varphi^{z,t,I}(\xi,\la)=\gel({\op D}^I_{(z,t)}\check\nu_{(z,t)}*u)(\xi,\la)
=\gel({\op D}^I_{(z,t)}\check\nu_{(z,t)})(\xi,\la)
\qquad \forall (\xi,\la)\in \cfan\cap B_\rho ,
$$
and by Proposition~\ref{distolo}
\begin{align*}
\gel({\op D}^I_{(z,t)}\check\nu_{(z,t)})(\xi,\la)
&=
\dualR{{\op D}^I_{(z,t)}\check\nu_{(z,t)}}{\check\Phi_{\xi,\la}}
={\op D}^I\left((w,s)\mapsto\int_K \Phi_{\xi,\la}(kw,s)\, dk \right)(z,t)
\\
&={\op D}^I\Phi_{\xi,\la}(z,t).
\end{align*} 
Finally, reasoning as before,
 \begin{align*}
\| \varphi^{z,t,I}\|_{C^k( B_\rho)}
&\leq 
C\,\|{\op D}^I_{(z,t)}\check\nu_{(z,t)}*u\|_{(M)}
\leq    C \, \big(1+|(z,t)|\big)^{M}
\qquad
\forall (z,t)\in \Hn.
\qedhere
\end{align*}
\end{proof}

Our characterization of the inverse spherical transform of compactly 
supported distributions is the following.

 \begin{theorem}\label{maininv}
Let $U$ be in $\cS'_0(\Sigma)$. The following conditions are equivalent.
\begin{enumerate}
\item
$\suppfan(U)$ is finite;

\item  
$\gel^{-1}U$ coincides with a smooth slowly growing function function
on $\Hn$ and for every $p$ in $[1,\infty]$ there exists $\beta>0$ such that
$$
\limsup_{j\to\infty}
\|\,(1+\cA)^{-\beta}\,L^j\gel^{-1}U\,\,\|_{L^p(\Hn)}^{1/j}
 <\infty;
$$
 
\item 
 for every large $j$  the distribution
$ L^j \gel^{-1}U$ coincides with a measurable function
on $\Hn$   and there exist $\beta>0$ and $p$ in $[1,\infty]$
such that
$$
\liminf_{j\to\infty} 
\|\,(1+\cA)^{-\beta}\,L^j\gel^{-1}U\,\,\|_{L^p(\Hn)}^{1/j}
<\infty .
$$

\end{enumerate}
Moreover, if any of these conditions is satisfied, then $\gel^{-1}U$
 is a smooth slowly growing function on $\Hn$
and for every $p$ in $[1,\infty]$ there exists $\beta>0$ such that
\begin{equation}\label{rhospettr}
\lim_{j\to\infty} \|\,(1+\cA)^{-\beta}\,L^j\gel^{-1}U\,\,\|_{L^p(\Hn)}^{1/j}= \suppfan(U).
\end{equation}
\end{theorem}

As in the previous section, we split the proof of our characterization into several parts.

 \begin{proposition}\label{inf}
 Let $U$ be in $\cS'_0(\cfan)$.   
Suppose that there exists  $J$ in $\N$ such that
  for every 
 $j\geq J$ the distribution $ L^j \gel^{-1} U$ is of the form $f_j\,\haar $, where $f_j$ is 
 a locally integrable function on~$\Hn$. Then for every $N$ in $\N$
and every $p$ in $[1,\infty]$
  $$
 \liminf_{j\to\infty} 
\|\,(1+\cA)^{-N}\,f_j\,\,\|_{L^p(\Hn)}^{1/j}
\geq \suppfan(U).
 $$
 \end{proposition}
 
 \begin{proof} For the 
 reader's convenience we write the proof although it  
  follows the lines of that of Proposition~\ref{due}.
We may suppose that  $\suppfan(U)$ is positive, because
in the case where $\suppfan(U)=0$, there is nothing to prove.
 
   Let  $\|\,(1+\cA)^{-N}\, f_j\,\,\|_{L^p(\Hn)}<\infty$.
 Suppose that $0<\eps < \suppfan(U)/2$ and let  $\psi$ be smooth function on $\R^2$
  with compact support in the set 
$$
\{(\xi,\la)\in \R^2\, :\, \suppfan(U)-\eps<\xi< \suppfan(U)+\eps \}
$$
 such that
$\dualR{U}{\psi}\neq 0$. 
For every nonnegative integer $j$, 
define a smooth function on $\R^2$
  with compact support by 
  $\psi_j(\xi,\la)=\xi^{-j}\, \psi(\xi,\la)$ for every $(\xi,\la)$ in $\R^2$.
  Then for $1\leq p\leq \infty$,
  \begin{align*}
|\dualR{U}{\psi}|
&=
|\dualR{\xi^j\,U}{ \psi_j}|
=
|\dualH{ {L^j \gel^{-1}U}}{(\cG\inv {\psi_j}_{|_\Sigma})\check{\phantom a}}|
\\
&\leq \|(1+\cA)^{-N}f_j\|_{L^p(\Hn)}\, \|(1+\cA)^{N}\cG\inv {\psi_j}_{|_\Sigma}
\|_{L^{p'}(\Hn)}
\end{align*}

Let $a$ be a positive integer such that 
$\|(1+\cA)^{N-a} 
\|_{L^{p'}(\Hn)}<\infty$, then by Lemma~\ref{potenzeM+}

  \begin{align*}
  \|(1+\cA)^{N}\,\cG\inv {\psi_j}_{|_\Sigma}
\|_{L^{p'}(\Hn)}
 & \leq
\|(1+\cA)^{N-a} 
\|_{L^{p'}(\Hn)} 
\| (1+M_+ )^a\,  {\psi_j} 
\|_{L^1(\cfan)}
\\ 
& \leq C_a
\,\big(\suppfan(U)+\eps\big)^{a}\, 
\sum_{s,r=1}^{2a}\|\partial_\la^s\partial_\xi^r \psi_j \|_{L^\infty(\R^2)} 
\\ & \leq 
\,C_a
\,\big(\suppfan(U)+\eps\big)^{a}\, \, j^{2a}\,\big(\suppfan(U)-\eps\big)^{-j}.
\end{align*}
  Therefore
  $$
  \|(1+\cA)^{-N}\,f_j\|_{L^p(\Hn)}
\geq 
 {|\dualR{U}{\psi}|}  \, C_{a,\eps}\,j^{-2a}\,\big(\suppfan(U)-\eps\big)^{j}
   $$
   and the thesis follows.
  Similar considerations can be used in the case where $\rho(U)=\infty$. 
 \end{proof}

 \begin{proposition}\label{limiteinv}
Let $U$ be in  $\cS'_0(\cfan)$ with $\suppfan(U)<\infty$. 
  Then    $\gel\inv U$
  coincides with a smooth slowly growing  function 
  $f$ on $\Hn$ and for every $p$ in $[1,\infty]$ there exists $h>0$ such that
$$
\limsup_{j\to\infty}
\|\,(1+\cA)^{-h}\,L^j f\,\,\|_{L^p(\Hn)}^{1/j}
 \leq \suppfan(U) .
$$
\end{proposition}

  \begin{proof}
  Since $\suppfan(U)<\infty$, the distribution~$U$ is compactly supported
  and therefore  $\gel\inv U$
  coincides with the smooth function $f_U$ on $\Hn$ 
  by Lemma~\ref{oloinv} and $f_U$ is slowly growing by Proposition~\ref{invtemp}. 
 Moreover, the estimate \pref{e:Lj} holds: if $\rho>\suppfan(U)$ 
and $k$ is the degree of $U$, 
there exists $M$ such that for all $j\geq 0$ 
$$
 \left|
     L^j f(z,t)
     \right|
     \leq C \,(1+j)^k\, \rho^j\, \big(1+|(z,t)|\big)^{M}.
$$
Let $p$ in $[1,\infty]$ be fixed and choose $h$ such that $(1+\cA)^{-h+M/2}$ is in $L^p(\Hn)$.
Then for every $\rho>\rho(U)$ 
\begin{align*}
\|\,(1+\cA)^{-h}\,L^j f\,\,\|_{L^p(\Hn)}
&\leq \|\,(1+\cA)^{-h+M/2}\,\|_{L^p(\Hn)} \, \|\,(1+\cA)^{-M/2}\,L^j f\,\,\|_{L^\infty(\Hn)}
\\
&\leq C\,(1+j)^k\, \rho^j
\end{align*}
so that 
$
\limsup_{j\to\infty}
\|\,(1+\cA)^{-h}\,L^j f\,\,\|_{L^p(\Hn)}^{1/j}
 \leq \rho 
$, for every  $\rho>\suppfan(U)$.
\end{proof}
 
\subsection{Square-integrable functions and Schwartz functions }

Reasoning as in the proof of Theorem~\ref{teoL2} and Corollary~\ref{L2} it easy to prove the following characterization for square-integrable functions.

\begin{theorem}\label{L2inv}
Let $\rho\geq 0$. Then  
$\gel$
is a bijection from the space 
 $$\{f\in\Lduerad 
\,\,: 
\lim_{j\to\infty} 
\|\,\,L^j f\,\,\|_{L^2(\Hn)}^{1/j}
\leq \rho
\}$$
onto the space $$\{F\in L^2(\cfan)
\,\,: \rho(F)
\leq \rho
\}\ .$$ 
\end{theorem}

In the case of Schwartz functions, we obtain the following results. 
For $F$ in $\cS(\cfan)$ we denote $\suppfan(F)=\suppfan(F\mu)$, so that
$$
\suppfan(F)=\sup\{\xi\, :\,F(\xi,\la)\not= 0\,\quad\text{and}\quad (\xi,\la)\in \cfan\}.
$$
 
 \begin{proposition}\label{liminvsch}
  Let $1\leq p\leq \infty$ and let $F$ be in $\cS(\cfan)$. 
  Then for every $h\geq 0$,
 $$
 \lim_{j\to\infty} 
\|\, \left(1+\cA \right)^{h} L^j \gel\inv F\,\,\|_{L^p(\Hn)}^{1/j}
= \suppfan(F).
 $$
\end{proposition}

 \begin{proof} Suppose $0<\suppfan(F)<\infty$.
 If $\gamma>0$ is big enough so that 
 $ \left(1+\cA \right)^{-\gamma}$ is in $L^p(\Hn)$ by Lemma~\ref{potenzeM+}  we obtain
   \begin{align*}
  \|\,\left(1+\cA \right)^{h} L^j \gel\inv F\,\,\|_{L^p(\Hn)}
  & \leq 
   \|\, \left(1+\cA \right)^{-\gamma}\,\|_{L^p(\Hn)}
   \|\, \left(1+\cA \right)^{h+ \gamma}\,L^j \gel\inv F\,\|_{L^\infty(\Hn)}
    \\
   & \leq  C\,
    \|\,  \left(1+M_+ \right)^{h+ \gamma}\big( \xi^j\, F\big)\|_{L^1(\cfan)}
\\
   & \leq
      C\, j^{2h+2 \gamma}\left(\suppfan(F)\right)^j.
\end{align*}
Hence 
$$
 \limsup_{j\to\infty} 
\|\, \left(1+\cA \right)^{h}\, L^j \gel\inv F\,\,\|_{L^p(\Hn)}^{1/j}
\leq \suppfan(F).
 $$
 and the thesis follows from Propostition~\ref{inf}.
  The cases $\suppfan(F)=0,\infty$ are trivial. 
 \end{proof}

\begin{theorem}\label{schinv}
Let $F$ be in $\cS(\cfan)$. The following conditions are equivalent.
\begin{enumerate}

\item  $\rho(F)$ is finite;

\item  for every $h\geq 0$ and every $p$ in $[1,\infty]$,
 $\limsup_{j\to\infty} 
\|\,\cA^h\,L^j\gel\inv F\,\,\|_{L^p(\Hn)}^{1/j}
$ is finite;

\item  there exists  $p$ in $[1,\infty]$ such that
 $\liminf_{j\to\infty} 
\|\,L^j\gel\inv F\,\,\|_{L^p(\Hn)}^{1/j}
$ is finite.

\end{enumerate}
Moreover, if any of these conditions is satisfied, then for every $h\geq 0$ and every $p$ in $[1,\infty]$,
$$
\lim_{j\to\infty} 
\|\, \left(1+\cA \right)^{h}\, L^j \gel\inv F\,\,\|_{L^p(\Hn)}^{1/j}
= \suppfan(F).
$$
\end{theorem}

\begin{proof} The implication $(1)\Rightarrow (2)$ follows by Propostion~\ref{liminvsch}. The implication
  $(2)\Rightarrow (3)$ is trivial. The implication $(3)\Rightarrow (1)$ follows by Propostion~\ref{inf}.   
  \end{proof}

\end{document}